\documentclass[12pt]{amsart}
\usepackage{amsmath,amssymb,amsbsy,amsfonts,latexsym,amsopn,amstext,cite,
                                               amsxtra,euscript,amscd,bm}
\usepackage{url}

\usepackage{mathrsfs}

\usepackage{color}
\usepackage[colorlinks,linkcolor=blue,anchorcolor=blue,citecolor=blue,backref=page]{hyperref}
\usepackage{color}
\usepackage{graphics,epsfig}
\usepackage{graphicx}
\usepackage{float}
\usepackage{epstopdf}
\hypersetup{breaklinks=true}

\usepackage[np]{numprint}
\npdecimalsign{\ensuremath{.}}

\usepackage{bibentry}

\usepackage[english]{babel}
\usepackage{mathtools}
\usepackage{todonotes}
\usepackage{url}

\usepackage[norefs,nocites]{refcheck}

      \def \l{\lambda}

\newtheorem{theorem}{Theorem}
\newtheorem{thm}[theorem]{Theorem}
\newtheorem{lem}[theorem]{Lemma}

\newtheorem{rem}[theorem]{Remark}

\numberwithin{equation}{section}
\numberwithin{theorem}{section}
\numberwithin{table}{section}
\numberwithin{figure}{section}

\def\squareforqed{\hbox{\rlap{$\sqcap$}$\sqcup$}}
\def\qed{\ifmmode\squareforqed\else{\unskip\nobreak\hfil
\penalty50\hskip1em \nobreak\hfil\squareforqed
\parfillskip=0pt\finalhyphendemerits=0\endgraf}\fi}

\newfont{\teneufm}{eufm10}
\newfont{\seveneufm}{eufm7}
\newfont{\fiveeufm}{eufm5}
\newfam\eufmfam
     \textfont\eufmfam=\teneufm
\scriptfont\eufmfam=\seveneufm
     \scriptscriptfont\eufmfam=\fiveeufm
\def\eps{\varepsilon}

\def\eqref#1{(\ref{#1})}

\def\le{\leqslant}
\def\leq{\leqslant}
\def\ge{\geqslant}
\def\leq{\leqslant}

\def\cS{{\mathcal S}}

\def\cW{{\mathcal W}}

 \def\0{{\mathbf{0}}}

\def\({\left(}
\def\){\right)}
\def\l|{\left|}
\def\r|{\right|}
\def\fl#1{\left\lfloor#1\right\rfloor}

\def\mand{\qquad \mbox{and} \qquad}

\usepackage{ucs}
\usepackage{ae}
\usepackage{microtype}
\usepackage[utf8x]{inputenc}
\usepackage[T1]{fontenc}
\usepackage[english]{babel}

\usepackage{amsmath}
\usepackage{amssymb}
\usepackage{amstext}
\usepackage{amsthm}

\usepackage{enumitem}

\usepackage{xcolor}

\usepackage{mathtools}

\usepackage[numbers]{natbib}

\newif\ifcomment
\newcommand{\comment}[1]{\ifcomment#1\fi}

\let\varepsilon\varepsilon

\begin{document}

\title[Smooth numbers with few non-zero digits]{Smooth numbers with few non-zero  binary digits}

\author[M. Hauck] {Maximilian Hauck}
\address{MH: Universit{\"a}t Bonn, Endenicher Allee 60, 53115 Bonn, Germany}
\email{max.hauck01@gmail.com}

\author[I. E. Shparlinski] {Igor E. Shparlinski}
\address{IES: Department of Pure Mathematics, University of New South Wales,
Sydney, NSW 2052, Australia}
\email{igor.shparlinski@unsw.edu.au}

\begin{abstract}    We use bounds of character sums and some combinatorial arguments to show the abundance of very smooth numbers which also have very few non-zero 
binary digits. 
\end{abstract}


\maketitle

\section{Introduction}

\subsection{Motivation and background} 

Since recently there  has been a lot of interest  in arithmetic properties of 
 integers with various digit restrictions in a given integer base $g$. 
This includes the work of Bourgain~\cite{Bou1, Bou2} and  Swaenepoel~\cite{Swa20} 
on primes with prescribed digits
on a positive proportion of positions in their digital expansion, the resolution of the Gelfond conjecture by Mauduit and Rivat~\cite{MaRi2},
and the results of Maynard~\cite{May19,May22} on primes with missing digits, see also~\cite{Bug1,BK,Col09,DES3,DMR,Kar22,Naslund, Pratt} and references therein.  

Prime divisors of integers with very few non-zero $g$-ary digits  have been studied in~\cite{Bou0, CKS, Shp08}.

A variety of results on primitive roots and quadratic nonresidues modulo a prime $p$,
which satisfy various digit restrictions can be found in the series of work~\cite{DES1,DES2,DES3}. 

Furthermore, Mauduit and Rivat~\cite{MaRi1} and Maynard~\cite{May22} have also studied values of integral polynomials with various digital restrictions.

We also note that special integers with restricted digits appear in the context of cryptography~\cite{ShparlinskiCharacterSums,Meng,ShparlinskiSmooth}. 

Here we consider some digital problems for smooth integers. We recall that by a result of 
Bugeaud and Kaneko~\cite[Theorem~1.1]{BK} any integer $N$ with at most $k\ge 3$ nonzero digits
in any fixed basis $g\ge 2$ not dividing $N$ has a prime divisor $p\mid N$ with 
\begin{equation}
\label{eq:super smooth}
p \ge \(\frac{1}{k-2}+o(1)\)   \frac{\log \log N\,   \log \log \log N}{\log\log \log \log N}
\end{equation}
as $N \rightarrow \infty$, see also~\cite{Bug2}.  The proof of~\eqref{eq:super smooth} uses some classical methods of Diophantine analysis, such as the bounds of linear forms in logarithms. Our technique is 
different and shows that there are many both reasonably smooth and sparse binary 
integers. 
\subsection{Smooth sparse integers} 
We recall that an integer $s$ is called $y$-smooth if it has no prime divisors $p > y$, see~\cite{Granville, Hildebrand} for some background. 

For a fixed absolute constant $\zeta\ge 1$, 
given a real number  $A>\zeta^3$ we define 
\begin{equation}
\label{eq:mu 0}
\mu_0(A)   =  \(1/2-1/2A\)  \(1 - \zeta  A^{-1/3}\)
\end{equation}
and define 
$\vartheta_0(A)<1/2$ by the equation
\begin{equation}
\label{eq:theta 0}
H\(\vartheta_0(A)\) =  1-\frac{1}{(1-\mu_0(A) )A }, 
\end{equation}
where $H(\rho)$ is the {\it binary entropy\/} function defined by
\begin{equation}
\label{eq:bin entrop}
    H(\rho)=-\rho\frac{\log\rho}{\log 2}-(
    1-\rho)\frac{\log(1-\rho)}{\log 2}\;.
\end{equation}
In particular, notice that $\vartheta_0(A)\rightarrow 1/2$ as $A\rightarrow\infty$.

We are interested in the existence of smooth integers with only few non-zero binary digits.
Clearly, this question makes sense only for odd integers as otherwise powers of $2$ give a 
perfect solution.


\begin{thm}
\label{thm:sparse1}
There is an  absolute constant $\zeta\ge 1$ such that  for  $A>\zeta^3$ 
the following holds: For any $\vartheta < \vartheta_0(A)$  and sufficiently large $n$, there exists an odd $n^A$-smooth $n$-bit integer  with at least 
$$
\(\mu_0(A) + \(1-\mu_0(A)\)\vartheta\)n
$$
zeros in its binary expansion.
\end{thm}

The proof of Theorem~\ref{thm:sparse1} is based on a refinement of the argument of~\cite[Theorem~6]{ShparlinskiSmooth} combined 
with the approach of~\cite{ShparlinskiCharacterSums}, which 
in turn relies on bounds for short multiplicative character sums from~\cite{IwaniecLSeries} (see also~\cite{BaSh}).
It also uses  a combinatorial argument, which originates from~\cite{ShparlinskiPrimPoly} and 
has been used in~\cite{DES1,DES2,Naslund}.  

The  constant $\zeta$ in~\eqref{eq:mu 0} is directly related to the constant in the bound on short 
character sums in~\cite{BaSh}, see Section~\ref{sec: proof T.1.1}; it is effective and can be 
explicitly evaluated.

Using another approach based on the observation that $2^k+1$ is quite smooth if we take $k$ as the product of the first $r$ odd primes, we obtain the following further result in the same direction:

\begin{thm}
\label{thm:sparse2}
Let $0<\alpha<1$ be fixed. 
Then there exist infinitely many $n$-bit integers $N\ge 1$ which are
$Y^{1+o(1)}$-smooth, where 
$$
Y = \exp\(2^{1/2}e^{-\gamma} \alpha^{-1/2} \(\log N\)^{1/2}\(\log\log \log N\)^{-1}\)
$$
and  $\gamma$ denotes the Euler-Mascheroni constant, 
which  have at most 
$$
\alpha n+O(n^{1/2}\log n)
$$
non-zero digits in their binary expansion. 
\end{thm}

It is easy to see from the proof of Theorem~\ref{thm:sparse2} that one can obtain 
a more uniform version of this result when both $\alpha$ and $N$ vary in such a way that $\alpha^{-1} = o(\log N)$.
We now make this observation more concrete and use a similar argument to produce another construction in a different regime of smoothness and sparseness.

\begin{thm}
\label{thm:sparse3}
Let $0\le \alpha \le 1$ be fixed. 
Then there exist infinitely many $n$-bit integers $N\ge 1$ which are
$Y^{1+o(1)}$-smooth where 
$$
Y = \exp\(2e^{-\gamma}\(\log 2)^{\alpha/2}(\log N\)^{1-\alpha/2}\(\log\log \log N\)^{-1}\)  
$$
and  $\gamma$ denotes the Euler-Mascheroni constant, 
which  have at most 
$$
\frac{1}{2\log 2}n^\alpha+O(n^{\alpha/2}\log n)
$$
non-zero digits in their binary expansion. 
\end{thm}

\subsection{Notation}

Throughout the paper, the notations $U = O(V)$, $U \ll V$ and $ V\gg U$  are equivalent to $|U|\leqslant c V$ for some positive constant $c$, which throughout the paper is absolute. If $U\ll V$ and $V\gg U$, we write $U\asymp V$.

Moreover, for any quantity $V> 1$ we write $U = V^{o(1)}$ (as $V \rightarrow \infty$) to indicate a function of $V$ which satisfies $ V^{-\eps} \le |U| \le V^\eps$ for any $\eps> 0$, provided $V$ is large enough. One additional advantage of using $V^{o(1)}$ is that it absorbs $\log V$ and other similar quantities without changing the whole expression.  

We also write $U \sim V$ as an equivalent of $(1-\varepsilon) V \le U \le (1+\varepsilon) V$
for any $\eps> 0$, provided $V$ is large enough. 

For a finite set $\cS$ we denote its cardinality by $\#\cS$.

For $n\in\mathbb{N}$, we denote the $n$-th cyclotomic polynomial by $\Phi_n$ and the Euler function by $\varphi$. Furthermore, we write $A_n$ for the largest absolute value of one of the coefficients of $\Phi_n$.

For a natural number $n$, we use $s_2(n)$ to denote the sum of the digits of $n$ in its binary representation.

We write $2=p_1<p_2<\dots$ for the prime numbers.

We denote by $\gamma\approx 0.577 $ the Euler-Mascheroni constant.

Finally, the expressions $1/ab$ mean $1/(ab)$ (which deviates from the canonical
interpretation $b/a$).
%

\section{Preparations}

\subsection{Arithmetic functions} 
 We  make use of the following well-known upper bound on the number $\tau(w)$ of divisors of a natural number $w$, see, for example, in~\cite[Equation~(1.81)]{IwaniecAnalyticNT}. 
\begin{lem}
\label{lem:divisorbound}
We have
$$
\tau(w)\leq w^{o(1)}\;.
$$
as $w \rightarrow \infty$. 
\end{lem}

Bateman~\cite{Bat} gives the following upper bound on the coefficients of cyclotomic polynomials:  

\begin{lem}
\label{lem:cyclotomiccoeffs}
We have 
$$
A_n\leq\exp\left(\frac{1}{2}\tau(n)\log n\right)
$$
for all $n\geq 1$.
\end{lem}


We also need to bound the Euler function of the product of the first odd primes.  

\begin{lem}
\label{lem:boundphi}
Let $k = p_2\cdots p_r$ for some integer $r\geq 2$. Then $$
\varphi(p_2\cdots p_r)=  2e^{-\gamma}\frac{k}{\log \log k}(1+o(1))
$$
as $r\rightarrow\infty$.
\end{lem}

\begin{proof}
This is a direct application of Mertens' so-called third theorem, see~\cite{Mertens}:
$$
\frac{\varphi(p_2\cdots p_r)}{p_2\cdots p_r}=  \prod_{i=2}^r \left(1-\frac{1}{p_i}\right) = 2 \prod_{i=1}^r \left(1-\frac{1}{p_i}\right)=\frac{ 2e^{-\gamma}}{\log p_r}(1+o(1)).
$$
Using that by the prime number theorem 
$$
k  = \exp\(\(1 + o(1)\)p_r\)
$$ 
we conclude the proof. 
\end{proof}

Finally, we need the following elementary result, see~\cite[Proposition~2.2]{HLS11}, 
which asserts that $s_2$ is  subadditive,
\begin{lem}
\label{lem:subadditive}
For any $m, n\in\mathbb{N}$, we have
$$
s_2(m+n)\leq s_2(m)+s_2(n).
$$
\end{lem}
 
\subsection{Counting smooth numbers} 
Let $\Psi(x, y)$ denote the number of $y$-smooth numbers less than or equal to $x$. We have the following result obtained. after simple manipulations, by combining~\cite[Theorem~3]{Hildebrand} 
and~\cite[Equation~(2.4)]{Hildebrand}:

\begin{lem}
\label{lem:psicxHildebrand}
For $x\geq y\geq 2$ and $1\leq c\leq y$, we have  
\begin{align*}
    \Psi(cx, y)=\Psi(x, y)& \left(1+\frac{y}{\log x}\right)^{\log c/\log y}\\
    & \qquad  \times \left(1+O\left(\frac{\log\log(1+y)}{\log^2 y}\cdot\log\left(1+\frac{y}{\log x} \right) \cdot \log c \right)\right)\\
    & \qquad  \qquad \qquad    \qquad \qquad  \qquad \qquad  \qquad \times \left(1+O\left(\frac{1}{u}+\frac{\log y}{y}\right)\right)
\end{align*}
uniformly, where $u=\log x/\log y$.
\end{lem}

We note that, for completeness, we have presented Lemma~\ref{lem:psicxHildebrand} in full generality, 
while we only use it for a fixed $c> 1$ and very small (compared to $x$) values of $y$. 
More precisely, we  use  Lemma~\ref{lem:psicxHildebrand} in the following form:  

\begin{lem}
\label{lem:psicx}
Fix real numbers $\alpha, \beta>0$, $c\geq 1$ and $A>1$. Then
$$
    \Psi(c\alpha x^\beta, \log^A x)\sim c^{1-1/A}\Psi(\alpha x^\beta, \log^A x)
$$
as $x\rightarrow\infty$. 
\end{lem}

\begin{proof}
Using Lemma~\ref{lem:psicxHildebrand} with $\alpha x^\beta$ in place of $x$ and $y=\log^A x$, we obtain
$$
    \Psi(cx^\beta, \log^A x) = (1+ o(1))  \Psi(x^\beta, \log^A x)\(1+\frac{\log^A x}{\beta\log x+\log \alpha}\)^{\log c/A\log\log x}.
$$
Finally, 
\begin{align*}
\(1+\frac{\log^A x}{\beta\log x+\log \alpha}\)^{\log c/A\log\log x} & =
\((1/\beta+o(1)) \log^{A-1}x\)^{\log c/A\log\log x} \\
& =  (1+o(1)) \(\log^{A-1}x\)^{\log c/A\log\log x} \:.
\end{align*}
Since $ \( \log^{A-1}x\)^{\log c/A\log\log x} = c^{1-1/A}$, the result follows.
\end{proof}

We also need an asymptotic formula for $\Psi(x, \log^A x)$, see for example~\cite[Equation~(1.14)]{Granville}:

\begin{lem}
\label{lem:psixlogax}
For any fixed $A>1$, we have
$$
    \Psi(x, \log^A x)=x^{1-1/A+o(1)}
$$
as $x\rightarrow\infty$.
\end{lem}
 
Finally, we need an upper bound of Harper~\cite{Har16}
on the number of smooth numbers in an arithmetic progression.  In fact we present it in a very 
special case tailored to our applications. Namely, let $\Psi(x, y; q,a)$ denote the number of $y$-smooth numbers less than or equal to $x$ in the residue class $a$ modulo $q$. 
By~\cite[Smooth Number Result~3, Section~2.1]{Har16} we have the following estimate.

\begin{lem}
\label{lem:psixlogax_aq}
For any fixed $A>1$,  $\beta>0$ and $\varepsilon >0$, we have
$$
    \Psi(x^\beta, \log^A x;q, a)\leq (x^\beta/q)^{1-1/A} x^{o(1)}  
$$
for all $a$ as $q \le x^{\beta-\varepsilon}$ and $x\rightarrow\infty$.
\end{lem}

\subsection{Sums involving binomial coefficients} 
We recall the definition~\eqref{eq:bin entrop}.
We frequently use the following result from~\cite[Chapter~10, Corollary~9]{Sloane}:

\begin{lem}
\label{lem:boundbincoeffs}
For any natural number $n$ and $0<\rho \le 1/2$, we have
$$
    \sum_{0\leq k\leq\rho n}\binom{n}{k}\leq 2^{nH(\rho)}.
$$
\end{lem} 

Furthermore, we  also need a bound on the product of binomial coefficient or, in other words, on the sum of their logarithms:

\begin{lem}
\label{lem:boundlogbincoeffs}
For any integer $n\geq 2$, we have
$$
\sum_{k=0}^n \log\binom{n}{k}=\frac{1}{2}n^2+O(n\log n).
$$
\end{lem}

\begin{proof}
By the Stirling formula, we have $\log n!=n\log n-n+O(\log n)$ for $n\geq 2$ and therefore
$$
\log\binom{n}{k}=n\log n-k\log k-(n-k)\log (n-k)+O(\log n)
$$
whenever $2\leq k\leq n-2$ and $n\geq 2$. Summing over all $k$, we see that
$$
\sum_{k=0}^n \log\binom{n}{k}=n^2\log n-2\sum_{k=1}^n k\log k+O(n\log n).
$$
For the remaining sum, partial summation yields
\begin{align*}
\sum_{k=1}^n k\log k&=\frac{n(n+1)}{2}\log n-\int_1^n \frac{\lfloor t\rfloor(\lfloor t\rfloor+1)}{2t}\,\mathrm{d}t \\
&=\frac{1}{2}n^2\log n-\int_1^n \frac{t}{2}\,\mathrm{d} t+O(n\log n) \\
&=\frac{1}{2}n^2\log n-\frac{1}{4}n^2+O(n\log n)
\end{align*}
and hence we obtain
$$
\sum_{k=0}^n \log\binom{n}{k}=\frac{1}{2}n^2+O(n\log n),
$$
as claimed.
\end{proof}

\subsection{Character sums modulo powers of $2$} 
From~\cite[Theorem 2.1]{BaSh}, we deduce the following estimate for character sums modulo powers of $2$:
\begin{lem}
\label{lem:charsums}
There exists an effective constant $\xi>0$ such that for all integers $k\geq 1$ and all non-principal characters $\chi$ modulo $q=2^k$, we have 
$$
\sum_{n=M+1}^{M+N} \chi(n)\ll N^{1-\xi\ell^2/k^2}
$$  
uniformly for all integers $M$ and $N=2^\ell$, where $\ell$ is an integer such that $\ell\leq k$, where implied constant is absolute.
\end{lem}

\begin{proof}
Assume $\chi$ is induced by the primitive character $\widetilde \chi$ modulo $\widetilde q=2^{k_0}$ with $k_0\geq 1$ and let $1\leq \widetilde N\leq \widetilde q$ such that $N\equiv \widetilde N\pmod{\widetilde q}$, that is, $\widetilde N=2^{\ell_0}$ with $\ell_0=\ell$ if $\ell\leq k_0$ and $\ell_0=k_0$ otherwise. Then 
$$
\sum_{n=M+1}^{M+N} \chi(n)=\sum_{n=M+1}^{M+N} \widetilde \chi(n)=\sum_{n=M+1}^{M+\widetilde N} \widetilde \chi(n)
$$
since complete sums of characters (of length $\widetilde q$ in this case) vanish.  
By~\cite[Theorem~2.1]{BaSh}, there is an effective constant $\xi_0>0$ such that
$$
\sum_{n=M+1}^{M+\widetilde N} \widetilde \chi(n)\ll {\widetilde N}^{1-\xi_0{\ell_0}^2/{k_0^2}}\;
$$
and the implied constant is absolute. Put $\xi=\min\{1/2, \xi_0\}$.

Assuming first that $\ell_0=\ell$, then clearly ${\widetilde N}^{1-\xi{\ell_0}^2/{k_0^2}}\leq N^{1-\xi\ell^2/k^2}$, so we are done. Now assume that $k_0<\ell$ and $\ell_0=k_0$. Then we have
$$
k_0(1-\xi)\leq k_0(1-\xi\ell^2/k^2) < \ell(1-\xi\ell^2/k^2)
$$
and therefore once again ${\widetilde N}^{1-\xi{\ell_0}^2/{k_0^2}}\leq N^{1-\xi\ell^2/k^2}$. 
\comment{
Finally, we are left with $\ell_0<c_0$. Then
$$
\sum_{n=M+1}^{M+\widetilde N} \widetilde \chi(n)\ll 2^{c_0}
$$
by the triangle inequality and due to $\ell(1-\xi\ell^2/k^2)\geq \ell(1-\xi)\geq c(1-\xi)=c_0$, this is again at most $N^{1-\xi\ell^2/k^2}$.
}
\end{proof}

\subsection{Smooth numbers with some bits prescribed} 
Using techniques from~\cite{ShparlinskiCharacterSums} and modifying the proof 
of~\cite[Theorem~6]{ShparlinskiSmooth}, we prove the following

\begin{lem}
\label{lem:prescribedbits} 
Let $A>1$ and $\varepsilon>0$ be fixed real numbers. with 
$2\varepsilon < 1-1/A$. Then for sufficiently large $n$ and any binary string $\sigma$ of length 
$$
m\leq n_0\(1-\(\frac{\(1-1/A\)\(1/A+2\varepsilon\)}{\xi\(1-1/A-2\varepsilon\)}\)^{1/3}\)\;,
$$
where $n_0=\lfloor(1/2-1/2A-\varepsilon)n\rfloor$ and $\xi$ is the constant from Lemma~\ref{lem:charsums}, there exist at least  
$2^{n(1-1/A)-m+o(n)}$ odd 
$n$-bit integers  which are $n^A$-smooth and have the bit pattern $\sigma$ at the positions $n_0-1, \dots, n_0-m$.
\end{lem}

\begin{proof}
Let $\mathcal{W}$ be the set of odd $n^A$-smooth numbers in the interval $(2^{(n-1)/2}, 2^{n/2}]$ and let $s$ denote the number defined by the binary string $\sigma$. We count the number $T(k)$ of solutions $w_1, w_2\in\mathcal{W}$ of $w_1w_2\equiv 2^{n_0-m}s+k\pmod{2^{n_0}}$ for $0\leq k<2^{n_0-m}$. Then the product $w_1w_2$ is $n^A$-smooth and has $n$ bits as well as the desired bit pattern.

Let $\mathcal{X}$ be the set of multiplicative characters modulo $2^{n_0}$. As in~\cite{ShparlinskiCharacterSums}, recalling the orthogonality relation
\begin{equation}
\label{eq:orthrels}
    \sum_{\chi\in\mathcal{X}} \chi(u)=\begin{cases} 0 & \text{if $u\not\equiv 1\pmod{2^{n_0}},$} \\
    2^{n_0-1} & \text{if $u\equiv 1\pmod{2^{n_0}}$,}
    \end{cases}
\end{equation}
see, for example,~\cite[Section~3.2]{IwaniecAnalyticNT}, we can express $T(k)$ as
$$
    T(k)=\frac{1}{2^{n_0-1}}\sum_{w_1,w_2\in\mathcal{W}}\sum_{\chi\in\mathcal{X}} \chi\((2^{n_0-m}s+k)w_1^{-1}w_2^{-1}\),
$$
where the inverses are taken modulo $2^{n_0}$ (recall that $w_1, w_2\in\mathcal{W}$ are odd). Also observe that $T(k)=0$ if $k$ is even.

If $k$ is odd, separating the contribution from the principal character  $\chi_0$ and changing the order of summation, we obtain
\begin{align*}
    T(k) &= \frac{(\#\mathcal{W})^2}{2^{n_0-1}}+\frac{1}{2^{n_0-1}}\sum_{\substack{\chi\in\mathcal{X} \\ \chi\neq \chi_0}} \chi(2^{n_0-m}s+k) \sum_{w_1,w_2\in\mathcal{W}} \chi\(w_1^{-1}w_2^{-1}\) \\
    &= \frac{(\#\mathcal{W})^2}{2^{n_0-1}} + \frac{1}{2^{n_0-1}}\sum_{\substack{\chi\in\mathcal{X} \\ \chi\neq \chi_0}}  \chi(2^{n_0-m}s+k) 
    \(\sum_{w\in\mathcal{W}} \chi\(w^{-1}\)\)^2 \;.
\end{align*} 
This yields
\begin{equation}
\label{eq:totalnumsols}
    \sum_{k=0}^{2^{n_0-m}-1} T(k)=\frac{(\#\mathcal{W})^2}{2^m}+\Delta\;,
\end{equation}
where, using $ \chi\(w^{-1}\) = \overline  \chi(w)$,   the error term $\Delta$ is given by
$$
    \Delta=\frac{1}{2^{n_0-1}}\sum_{\substack{\chi\in\mathcal{X} \\ \chi\neq \chi_0}}\sum_{k=0}^{2^{n_0-m}-1} \chi(2^{n_0-m}s+k)\(\sum_{w\in\mathcal{W}} \overline{\chi}(w)\)^2\;.
$$

Now we estimate $\Delta$, again following~\cite{ShparlinskiCharacterSums}. Namely, we have
$$
    |\Delta|\leq \frac{1}{2^{n_0-1}}\sum_{\substack{\chi\in\mathcal{X} \\ \chi\neq \chi_0}} \left|\sum_{k=0}^{2^{n_0-m}-1} \chi(2^{n_0-m}s+k)\right|\left|\sum_{w\in\mathcal{W}} \chi(w)\right|^2
$$
by the triangle inequality.
By Lemma~\ref{lem:charsums}, we have the estimate 
$$
    \sum_{k=0}^{2^{n_0-m}-1}\chi(2^{n_0-m}s+k)\ll 2^{(n_0-m)(1-\xi(n_0-m)^2/n_0^2)}
$$
for any non-principal character $\chi\in\mathcal{X}$ and $n_0$ sufficiently large. Moreover, by our choice of $m$ and $n_0$, we may further estimate
\begin{align*}
n_0-m& \geq \frac{n_0}{\xi^{1/3}}\left(\frac{\(1-1/A\)\(1/A+2\varepsilon\)}{1-1/A-2\varepsilon}\right)^{1/3} \\
&=\frac{n_0}{\xi^{1/3}}\left(\frac{1}{1-1/A-2\varepsilon}-1\right)^{1/3}(1-1/A)^{1/3} \\
&\geq\frac{n_0}{\xi^{1/3}}\left(\frac{n}{2n_0}-1+o(1)\right)^{1/3}(1-1/A)^{1/3}\\
&=\frac{n_0^{2/3}}{\xi^{1/3}}(n/2-n_0)^{1/3}(1-1/A)^{1/3}+o(n)
\end{align*} 
and therefore the estimate above becomes
$$
    \sum_{k=0}^{2^{n_0-m}-1}\chi(2^{n_0-m}s+k)\ll 2^{n_0-m-(n/2-n_0)(1-1/A)+o(n)}\;.
$$

Thus, we can further estimate
\begin{align*}
    |\Delta|&\ll 2^{-m-(n/2-n_0)(1-1/A)+o(n)}\sum_{\substack{\chi\in\mathcal{X} \\ \chi\neq \chi_0}} \left|\sum_{w\in\mathcal{W}} \chi(w)\right|^2\\ 
    &\ll 2^{-m-(n/2-n_0)(1-1/A)+o(n)}\sum_{\chi\in\mathcal{X}} \left|\sum_{w\in\mathcal{W}} \chi(w)\right|^2 \\
    &= 2^{-m-(n/2-n_0)(1-1/A)+o(n)}\sum_{\chi \in \mathcal{X}} \sum_{w_1,w_2 \in \mathcal{W}} \chi(w_1w_2^{-1})\;.
\end{align*}
 To estimate the last sum, we  change the order of summation and use the orthogonality relations~\eqref{eq:orthrels}. Namely, by Lemma~\ref{lem:psixlogax_aq}, there are at most $2^{(n/2-n_0)(1-1/A)+o(n)}$ elements of $\cW$ in any given residue class modulo $2^{n_0}$, so we obtain
\begin{align*}
 \sum_{w_1,w_2 \in \mathcal{W}}\sum_{\chi \in \mathcal{X}} \chi(w_1w_2^{-1})&=2^{n_0-1}\cdot\#\{(w_1, w_2)\in\mathcal{W}^2: w_1\equiv w_2\pmod{2^{n_0}}\} \\
 & \leq 2^{n_0+(n/2-n_0)(1-1/A)+o(n)} \#\mathcal{W}\;.
\end{align*}
Therefore, we overall get
$$
|\Delta|\ll 2^{n_0-m+o(n)}\#\mathcal{W}\;.
$$

To compare the error term with the main term, we need to determine $\#\cW$. To do this, observe that for any $x$, there is a bijection
\begin{align*}
\{w\in [1, x/2]:~w\text{ $n^A$-smooth}\}&\overset{1\,:\,1}{\longleftrightarrow}\{w\in [1, x]:~w\text{ even, $n^A$-smooth}\} \\
w&\mapsto 2w
\end{align*}
and thus the number of odd $n^A$-smooth numbers up to $x$ is given by $\Psi(x, n^A)-\Psi(x/2, n^A)$. Therefore, Lemma~\ref{lem:psicx} yields
\begin{align*}
    \#\mathcal{W}&=\left(\Psi(2^{n/2}, n^A)-\Psi(2^{n/2-1}, n^A)\right) \\
    &\hspace{2cm}-\left(\Psi(2^{(n-1)/2}, n^A)-\Psi(2^{(n-1)/2-1}, n^A)\right) \\
    &\sim \Psi(2^{(n-1)/2-1}, n^A)\left(2^{3(1-1/A)/2}-2^{1-1/A}-2^{(1-1/A)/2}+1\right)
\end{align*}
and due to
$$
    2^{3(1-1/A)/2}-2^{1-1/A}-2^{(1-1/A)/2}+1=(2^{(1-1/A)/2}-1)^2(2^{(1-1/A)/2}+1)\;,
$$
this is a positive proportion of $\Psi(2^{(n-1)/2-1}, n^A)$. Moreover, according to  Lemma~\ref{lem:psixlogax},
\begin{align*}
    \Psi(2^{(n-1)/2-1}, n^A)& \geq \Psi(2^{(n-3)/2}, (n-3)^A\log^A 2/2^A)\\
    & =  2^{(n-3)(1-1/A+o(1))/2}=2^{n(1-1/A+o(1))/2}\;.
\end{align*}
Combining both results, we deduce that
\begin{equation}
\label{eq:boundW}
    \#\mathcal{W}\geq 2^{n(1-1/A+o(1))/2}\;.
\end{equation}

Therefore, the ratio of the error term to the main term in~\eqref{eq:totalnumsols} can be bounded by
\begin{align*}
    \frac{|\Delta|}{(\#\mathcal{W})^2/2^m}&\ll \frac{2^{n_0-m+o(n)}\#\mathcal{W}}{(\#\mathcal{W})^2/2^m}=\frac{2^{n_0+o(n)}}{\#\mathcal{W}}\leq \frac{2^{n_0+o(n)}}{2^{n(1-1/A)/2}} \\
    &\leq \frac{2^{n(1-1/A-2\varepsilon)/2+o(n)}}{2^{n(1-1/A)/2}}=2^{-\varepsilon n +o(n)}=o(1)\;.
\end{align*}
Thus, putting this result back into~\eqref{eq:totalnumsols}, we obtain
$$
    \sum_{k=0}^{2^{n_0-m}-1} T(k)=\frac{(\#\mathcal{W})^2}{2^m}(1+o(1))\;.
$$
Using~\eqref{eq:boundW} again, we see that this becomes
$$
    \sum_{k=0}^{2^{n_0-m}-1} T(k)\geq \frac{2^{n(1-1/A+o(1))}}{2^m}(1+o(1))=\frac{2^{n(1-1/A+o(1))}}{2^m}\;.
$$

To conclude, it still remains to check how many pairs $(w_1, w_2)$ yield the same product $w=w_1w_2$. However, any such product $w$ satisfies $w\leq 2^n$ by construction and hence, by  Lemma~\ref{lem:divisorbound}, it can occur at most $2^{o(n)}$ 
times.
 Therefore, the number of pairwise distinct products is at least
$$
 2^{o(n)}   \sum_{k=0}^{2^{n_0-m}-1} T(k) \geq \frac{2^{n(1-1/A+o(1))}}{2^m}\;,
$$
as claimed. 
\end{proof}

\comment{
\begin{rem}
\label{rem:odd}
Note that all the smooth numbers constructed in the proof are odd, hence the statement of the theorem still holds true if we confine ourselves to odd smooth numbers.
\end{rem}
}
 
\comment{
\begin{rem}
\label{rem:improvement}
Note that using the given in~\cite{BaSh} improvement of the previous bound of Iwaniec~\cite[Lemma~6]{IwaniecLSeries}, 
one can replace $n^{3/4}\log n$ with $Cn^{2/3}$ for some absolute constant $C$
in the restriction on $m$ of Lemma~\ref{lem:prescribedbits}.
\end{rem} 
}

\section{Proof of Theorem~\ref{thm:sparse1}}  
\label{sec: proof T.1.1} 
Let $\varepsilon>0$. 
We set
$$
    n_0=\fl{(1/2-1/2A-\varepsilon)n} \mand     m=\fl{n_0\(1-\(\frac{\(1-1/A\)\(1/A+2\varepsilon\)}{\xi\(1-1/A-2\varepsilon\)}\)^{1/3}\)}\;.
$$
It is also convenient to define 
$$
\mu = m/n
$$
and note that there are some constants $c_2>  c_1>0$ depending only on $A$, but not on $\varepsilon$ 
such that for a sufficiently small $\varepsilon$,
$$
 \(\xi A\)^{-1/3}
+c_1\varepsilon \le 
 \(\frac{\(1-1/A\)\(1/A+2\varepsilon\)}{\xi\(1-1/A-2\varepsilon\)}\)^{1/3} \le \(\xi A\)^{-1/3} 
+c_2\varepsilon. 
$$
Note that the positivity of $c_1$ is crucial for us and in particular, this implies that 
for some constants $C_2>  C_1>0$ depending only on $A$, we have 
\begin{equation}
\label{eq:mu mu0}
 \mu_0(A) - C_2\varepsilon\le  \mu  
 \le \mu_0(A) -C_1\varepsilon , 
\end{equation}
where $\mu_0(A)$ is given by~\eqref{eq:mu 0} with $\zeta = \xi^{-1/3}$. 

Applying  Lemma~\ref{lem:prescribedbits} with the above values of $n_0$ and $m$, 
we see that the number $T$ of odd $n$-bit $n^A$-smooth numbers with a string of $m$ zeros at the positions $n_0-1, \dots, n_0-m$ is at least
\begin{equation}
\label{eq:boundN}
    T\geq\frac{2^{n(1-1/A+o(1))}}{2^m}=2^{(n-m) (1-1/(1-\mu)A)+o(n)}  . 
\end{equation}

Now consider the $n-m$ bits we have not prescribed yet. If at most a proportion $\rho<\vartheta_0(A)$ of them are zeros, where $\vartheta_0(A)$ is given by~\eqref{eq:theta 0}, then, by  
Lemma~\ref{lem:boundbincoeffs}, we can have at most 
$$
    \sum_{0\leq k\leq\rho(n-m)} \binom{n-m}{k}\leq 2^{(n-m)H(\rho)}
$$
different integers. However, since $\rho<\vartheta_0(A)$, we know that
$$
	H(\rho)< 1-\frac{1}{(1-\mu_0(A) )A } <1-\frac{1}{(1-\mu )A }.
$$
Recalling~\eqref{eq:boundN}, we conclude that  
$$
2^{(n-m)H(\rho)} < T
$$
if $n$ is  sufficiently large.

Thus, at least one of the $T$ odd $n$-bit $n^A$-smooth numbers not only has a string of $m$ consecutive zeros, which is by construction, but also has at least a proportion of $\rho$ zeros among the remaining $n-m$ digits. That is,
we see from~\eqref{eq:mu mu0}  that  its binary expansion contains at least
\begin{align*}
    m+\rho(n-m)& = (1-\rho)m +\rho n = n\(\mu (1-\rho)+ \rho\)\\
       & \ge n\(\mu_0(A) (1-\rho)+\rho - C_2 (1-\rho) \varepsilon\)\\
    & \ge n\(\mu_0(A) + \rho(1-\mu_0(A)) - C_2 (1-\rho) \varepsilon\)
\end{align*}
zeros. Choosing $\rho>\vartheta$ concludes the proof since $\varepsilon$ is arbitrarily small.

\section{Proof of Theorem~\ref{thm:sparse2}}

Put $k=p_2\cdots p_r$ for some $r>2$ and consider $m=2^k+1$.
We assume that $r \rightarrow \infty$, so we also have $k \rightarrow \infty$. 

We have the following factorisation into cyclotomic polynomials:
$$
m=2^k+1=-\prod_{d\mid k} \Phi_d(-2)
$$
and we can bound the factors using Lemma~\ref{lem:cyclotomiccoeffs} and Lemma~\ref{lem:divisorbound} as
\begin{align*}
|\Phi_d(-2)|&\ll A_d 2^{\varphi(d)}\leq \exp\left(\frac{1}{2}\tau(d)\log d\right)2^{\varphi(d)} \\
&\ll \exp\left(k^{o(1)}\log k\right)2^{\varphi(d)}\ll 2^{\varphi(d)+k^{o(1)}}\leq 2^{\varphi(k)+k^{o(1)}}.
\end{align*}
Moreover, due to our choice of $k$, Lemma~\ref{lem:boundphi} implies 
$$
|\Phi_d(-2)|\ll 2^{( 2e^{-\gamma}+o(1))k/\log \log k}<m^{( 2e^{-\gamma}+o(1))/\log \log k}=m^{( 2e^{-\gamma}+o(1))/\log\log \log m}.
$$
In particular, $m$ is $m^{(2e^{-\gamma}+o(1))/\log\log\log m}$-smooth.

Now, we consider the number 
\begin{equation}
\label{eq:def N}
N=m^{\ell},
\end{equation} 
where $\ell=\lfloor \beta k\rfloor$ for $\beta=2\alpha\log 2$. 
Without loss of generality we assume that $k$ is  large enough so $\ell \ge 2$.
Thus $N$ has 
\begin{equation}
\label{eq:ell and n}
n=k\ell+O(\ell)=\beta^{-1} \ell^2+O(\ell) 
\end{equation} 
bits, provided that $k \rightarrow \infty$.
Due to
$$
\ell\sim\beta k\sim \frac{\beta\log m}{\log 2}=\frac{2\alpha\log N}{\ell} 
$$
we also  have $\ell\sim (2\alpha\log N)^{1/2}$.
In particular,
\begin{align*}
\log \log \log N & = \log \log\(\ell   \log m\)  \le   \log \log \( \(2 \log m\)^2\) \\
& =  \log\( 2 \(\log \log m + O(1) \)\) = \(1+o(1)\)\log \log \log m.
\end{align*}
%
Therefore, we see that
\begin{align*}
m^{( 2e^{-\gamma}+o(1))/\log\log\log m}&=N^{( 2e^{-\gamma}+o(1))/(\ell \log\log \log N)} \\
&=N^{\(2e^{-\gamma}+o(1)\)/\((2\alpha\log N)^{1/2}\log\log \log N\)} = Y^{1+o(1)}, 
\end{align*}
that is, $N$ is $Y^{1+o(1)}$-smooth.

We now 
estimate the sparsity of $N$.
 By the binomial theorem, we know that
$$
N=(2^k+1)^\ell=\sum_{j=0}^\ell \binom{\ell}{j}2^{kj}.
$$
Noting that $\binom{\ell}{j}$ can have at most $\log\binom{\ell}{j}/\log 2+1$ binary digits and using the subadditivity of the sum of binary digits guaranteed by Lemma~\ref{lem:subadditive}, we see that Lemma~\ref{lem:boundlogbincoeffs} shows that the total number of non-zero digits of $N$ is at most
\begin{align*}
\frac{1}{\log 2}\sum_{j=0}^\ell \log\binom{\ell}{j}+(\ell+1)&=\frac{1}{2\log 2}\ell^2+O(\ell\log\ell) \\
&=\alpha n+O\(n^{1/2}\log n\),
\end{align*}
where we used that 
$\ell^2=\beta n+O\(n^{1/2}\)$  
due to~\eqref{eq:ell and n} in the last step. 
This proves the claim.

\section{Proof of Theorem~\ref{thm:sparse3}}

We proceed exactly as in the proof of Theorem~\ref{thm:sparse2} up to the point where we defined $N$ in~\eqref{eq:def N}. If $\alpha=0$, we choose $\ell=1$ and are done; otherwise, let $\ell=\lfloor k^\beta\rfloor$ for 
$$
\beta=\frac{\alpha}{2-\alpha}.
$$

Without loss of generality we assume that $k$ is  large enough so $\ell \ge 2$.
Thus $N$ has 
\begin{equation}
\label{eq:ell and n no2}
n=k\ell+O(\ell)=\ell^{1+1/\beta}+O\(\ell^{1/\beta}\) 
\end{equation}  
bits, provided that $k \rightarrow \infty$.
Due to
$$
\ell\sim k^\beta \sim \(\frac{\log m}{\log 2}\)^\beta= \(\frac{\log N}{\ell\log 2} \)^\beta
$$
we also  have $\ell\sim (\log N/\log 2)^{\beta/(1+\beta)}$.
In particular,
\begin{align*}
\log \log \log N & = \log \log\(\ell   \log m\)  \le   \log \log \( \(2 \log m\)^{1+\beta}\) \\
& =  \log\( (1+\beta) \(\log \log m + O(1) \)\) \\
&= \(1+o(1)\)\log \log \log m.
\end{align*}
%
Therefore, we see that
\begin{align*}
m^{( 2e^{-\gamma}+o(1))/\log\log\log m}&=N^{( 2e^{-\gamma}+o(1))/(\ell \log\log \log N)} \\
&=N^{\( 2e^{-\gamma}+o(1)\)/\((\log N/\log 2)^{\beta/(1+\beta)}\log\log \log N\)} = Y^{1+o(1)}
\end{align*}
that is, $N$ is $Y^{1+o(1)}$-smooth.  

As before, we see that the number of non-zero digits of $N$ is at most  
\begin{align*}
\frac{1}{2\log 2}\ell^2+O(\ell\log\ell)&=\frac{1}{2\log 2}n^{2\beta/(1+\beta)}+O(n^{\beta/(1+\beta)}\log n) \\
&=\frac{1}{2\log 2}n^\alpha+O(n^{\alpha/2}\log n)
\end{align*}
due to $\ell=n^{\beta/(1+\beta)}+O(1)$ by~\eqref{eq:ell and n no2} and this concludes the proof.


 \section{Comments}
 
 It is also interesting to study smooth values amongst balanced integers, that is, positive integers with equally many ones and zeros in their binary expansion.
 
 Using simple counting arguments, one can show that for any integer $n$, there exist $2n$-bit balanced 
 integers which are $Y$-smooth, where 
\begin{equation}
\label{eq:balanced}
Y = 2^{2n -  \(1/\sqrt{\pi} + o(1)\)n^{1/2}}. 
\end{equation}
Indeed, first we observe that by the Stirling formula, the number of  $2n$-bit balanced  integers is given by
\begin{equation}
\label{eq:mid binom}
\binom{2n-1}{n-1} =  \(1/2\sqrt{\pi} + o(1)\)\frac{4^n}{n^{1/2}}.
\end{equation}
 However, for any $u\in [1,2]$,
 $$
 \Psi(x, x^{1/u}) = \(1 -\log u\) x + O(x/\log x), 
 $$
 see~\cite[Equations~(1.3) and~(1.8)]{Granville}.
 Thus, if $u =1 +\eta$ with $\eta\asymp (\log x)^{-1/2}$, then the number of non-$x^{1/u}$-smooth 
 numbers $N \le x$ can be estimated as 
\begin{equation}
\label{eq:non-smooth}
 x-  \Psi(x, x^{1/u}) = (1+ o(1))\eta x.  
 \end{equation}
 Hence, for a sufficiently small $\varepsilon > 0$  with $\eta = \(1/2\sqrt{\pi}-\varepsilon\)n^{-1/2}$ and sufficiently large $n$, we 
 obtain from~\eqref{eq:mid binom} and~\eqref{eq:non-smooth} that
  $$
 4^n -  \Psi\(4^n, 4^{n/\(1+(1/2\sqrt{\pi} -\varepsilon)n^{-1/2}\)}\)  < \binom{2n-1}{n-1} ,
 $$
 and since  $\varepsilon > 0$  is arbitrary, we see that we can take $Y$ as in~\eqref{eq:balanced}.
 
Furthermore, with a similar technique as in the proofs of Theorem~\ref{thm:sparse2} 
and Theorem~\ref{thm:sparse3}, one can give an explicit construction of infinitely many rather smooth balanced numbers. 
Namely, taking 
$$k_1=p_2\cdots p_r \mand k_2=p_3\cdots p_r=k_1/3,
$$ 
the same argument as above shows that  $2^{k_1}+1$ is $2^{(2e^{-\gamma}+o(1))k_1/\log\log k_1}$-smooth and $2^{k_2}-1$ is $2^{(2e^{-\gamma}+o(1))k_2/\log\log k_2}$-smooth. Thus, the number $N=(2^{k_1}+1)(2^{k_2}-1)$ is $2^{(2e^{-\gamma}+o(1))k_1/\log\log k_1}=N^{(3e^{-\gamma}/2+o(1))/\log\log\log N}$-smooth and it has $n = k_1+k_2=4k_2$ total digits, exactly $n/2$ of which are ones.

 \section*{Acknowledgements}
 
 The authors are very grateful to Regis de la Bret{\`e}che for suggesting to use results from~\cite{Har16}
 and to the referee for the very careful reading of the manuscript. 
 
 This work started during a very enjoyable visit by the authors  to the 
Max Planck Institute for Mathematics, Bonn, whose hospitality and 
support is very much appreciated. 
 
During the preparation of this work I.S. was also supported in part by the  
Australian Research Council Grant  DP200100355.

\bibliographystyle{alpha}

\end{document}